\documentclass[10pt]{amsart}

\usepackage{amsfonts}
\usepackage{color}
\usepackage{amscd}
\usepackage{epsfig}
\usepackage{graphicx}
\usepackage{amsmath,amssymb}
\usepackage{amsthm}
\usepackage{url}
\usepackage{enumerate}

\newtheorem{thm}{Theorem}[section]
\newtheorem{lem}[thm]{Lemma}

\newtheorem{meth}[thm]{Method}
\newtheorem{prop}[thm]{Proposition}
\newtheorem{cor}[thm]{Corollary}

\newtheorem{rmk}[thm]{Remark}

\newtheorem{prob}[thm]{Problem}
\newtheorem{ex}[thm]{Example}

\def\O{{\mathcal O}}

\def\I{{\mathcal I}}

\def\F{{\mathcal F}}

\def\Z{{\mathbb Z}}

\def\Pthree{{\mathbb P}^3}

\def\Pic{\mathop{\rm Pic}}

\def\Cl{\mathop{\rm Cl}}

\def\Spec{\mathop{\rm Spec}}

\def\mod{\mathop{\rm mod}}

\def\fm{\mathfrak m}

\def\ord{\mathop{\rm ord}}

\def\supp{\mathop{\rm Supp}}

\def\ker{\mathop{\rm Ker}}

\def\ra{\rightarrow}

\newcommand{\ubm}[2]{\underbrace{#1}_{#2}}
\newcommand{\cyc}[1]{\langle {#1} \rangle}

\def\rdA{\mathbf A}
\def\rdD{\mathbf D}
\def\rdE{\mathbf E}

\title{Picard groups of normal surfaces}

\author{John Brevik}

\address{California State University at Long Beach, 
Department of Mathematics and Statistics, Long Beach, CA 90840}

\email{jbrevik@csulb.edu}

\author{Scott Nollet}

\address{Texas Christian University, Department of Mathematics, 
Fort Worth, TX 76129}

\email{s.nollet@tcu.edu}


\begin{document}
\bibliographystyle{plain}

\begin{abstract} 
We study the fixed singularities imposed on members of a 
linear system of surfaces in $\mathbb P^3_{\mathbb C}$ by its base locus $Z$. 
For a 1-dimensional subscheme $Z \subset \mathbb P^3$ with finitely many 
points $p_i$ of embedding dimension three and $d \gg 0$, we determine the 
nature of the singularities $p_i \in S$ for general 
$S \in |H^0 (\mathbb P^3, I_Z (d))|$ and give a method to compute 
the kernel of the restriction map $\Cl S \to \Cl \O_{S,p_i}$. 
One tool developed is an algorithm to identify the type of an $\rdA_n$  
singularity via its local equation. 
We illustrate the method for representative $Z$ and use Noether-Lefschetz 
theory to compute $\Pic S$. 
  
\end{abstract}

\maketitle

\section{Introduction}

The problem of computing the Picard groups of surfaces 
$S \subset \mathbb P^3_{\mathbb C}$ has a long history. 
The solution for smooth quadric and cubic surfaces was known 
in the 1800s in terms of lines on these surfaces. In the 1880s 
Noether suggested what happens in higher degree, but it wasn't 
until the 1920s that Lefschetz proved the famous result bearing 
their names: the very general surface $S$ of degree $d > 3$ has 
Picard group $\Pic S \cong \Z$, generated by the hyperplane section $H$.
Here {\it very general} refers to a countable intersection 
of Zariski open subsets. To produce typical families of surfaces $S$ 
with $\Pic S$ not generated by $H$, Lopez proved that very 
general surfaces $S$ of high degree containing a smooth connected 
curve $Z$ have Picard group freely generated by $H$ 
and $Z$ \cite[II, Thm. 3.1]{L}, a geometrically pleasing result 
with many applications \cite{CDE,CL,DE,FL}. 

Recently we extended these results, proving that the class group $\Cl S$ 
of the very general surface $S$ containing an arbitrary 1-dimensional subscheme $Z$ 
with at most finitely many points of embedding dimension three
\footnote{This is the weakest condition allowing $S$ to be a normal surface, so that $\Cl S$ is defined.} is freely generated by $H$ and the supports of the curve 
components of $Z$ \cite[Thm. 1.1]{BN}. 
This allows access to the Picard group via the exact sequence of Jaffe \cite[Prop. 3.2]{J1} (see also \cite[Prop. 2.15]{GD})
\begin{equation}\label{jaffe}
0 \to \Pic S \to \Cl S \to \bigoplus_{p \in {\rm {Sing}} S} \Cl \O_{S,p}
\end{equation}
provided we can find the kernels of the restriction maps 
$\Cl S \to \Cl \O_{S,p}$ at the singular points $p \in S$,  
where $\Cl \O_{S,p}$ is the divisor class group of the local ring. 
The answer being known at singular points of $S$ where $Z$ has embedding dimension $\leq 2$ \cite[Prop. 2.2]{BN}, our motivating question becomes:  

\begin{prob}\label{uno}{\em 
For $Z \subset \mathbb P^3$ and $p \in Z$ a point of embedding 
dimension three, find the kernel of the restriction 
map $\Cl S \to \Cl \O_{S,p}$. 
\em}\end{prob}

A general solution to Problem \ref{uno} is out of reach because one would need 
to classify all points of embedding dimension three points $p_i$ on curves $Z$ 
to state an answer. Instead we give a method of attack on the problem: 

\begin{meth}\label{madness}{\em 
The kernel of the restriction $\Cl S \to \Cl \O_{S,p}$ can be computed 
as follows. 
\begin{enumerate}
\item[Step 1.] The natural map $\Cl \O_{S,p} \to \Cl \widehat \O_{S,p}$ being
injective, we consider the composite map 
$\Cl S \to \Cl \O_{S,p} \hookrightarrow \Cl \widehat \O_{S,p}$, where power series tools are available.  
\item[Step 2.] Working in $\widehat \O_{S,p}$, use analytic coordinate changes to recognize the form of the singularity and compute the local class group 
$\Cl \widehat \O_{S,p}$ when possible. 
\item[Step 3.] Since $\Cl S$ is freely generated by $H$ and the supports of 
the curve components of $Z$ \cite[Thm 1.1]{BN}, it is enough to find the images of 
those supports for curve components of $Z$ passing through $p$ in 
$\Cl \widehat \O_{S,p}$ (the rest map to zero). 
\end{enumerate}
\em}\end{meth}

\begin{rmk}\label{recognition}{\em
Method \ref{madness} can {\it always be carried out if $Z$ is 
locally contained in two smooth surfaces meeting transversely at $p$}.
This is because the analytic local equation of $S$ at $p$ contains an 
$xy$ term and we can employ our {\it recognition theorem}: 
Theorem \ref{ratrecog} gives an inductive algorithm 
that recognizes an $\rdA_n$ singularity in at most $n$ steps, but 
finishes in just $1$ step with probability $1$. While most of this paper is devoted to 
examples illustrating Method \ref{madness}, Theorem \ref{ratrecog} may be the 
most useful general result presented here.
\em}\end{rmk}

\begin{rmk}\label{class}{\em
Regarding Step 3, the images of the 
supports of the curve components of $Z$ containing $p$ generate 
$\Cl O_{S,p}$ as a subgroup of $\Cl \widehat O_{S,p}$ \cite[Prop. 2.3]{BN2}, 
at least for very general $S$.
This gives a geometric way to see the class group of a ring in its completion. 

(a) In particular, the map $\Cl S \to \Cl \O_{S,p}$ is zero if $p$ is an 
isolated point of $Z$, since $Z$ has no curve components passing through $p$. 
Therefore only the 1-dimensional part of $Z$ contributes to the answer. 

(b) Srinivas has asked \cite[Ques. 3.1]{srinivas} which subgroups appear as 
$\Cl B \subset \Cl A$ where $B$ is a local $\mathbb C$-algebra with $A = \widehat B$. 
We proved that for complete local rings $A$ corresponding to the rational double 
points $\rdA_n, \rdD_n, \rdE_6, \rdE_7, \rdE_8$, the answer is 
{\it every} subgroup \cite[Thm. 1.3]{BN2}. 
We constructed the rings $B$ as the geometric local rings $\O_{S,p}$ arising from 
general surfaces $S \subset \mathbb P^3$ containing a fixed base locus forcing the singularity at $p$. 
This poses a stark contrast to results of Kumar \cite{K}, who showed that if $B$ has 
fraction field $\mathbb C (x,y)$ and the singularity type is $\rdE_6, \rdE_7$ or $\rdA_n$ 
with $n \neq 7,8$, then $B$ is determined by $A$ and hence $\Cl B = \Cl A$. 
\em}\end{rmk}

We illustrate Method \ref{madness} by giving complete answers for the following 
base loci $Z$:

\begin{enumerate}
\item[(1)] Unions of two multiplicity structures near $p$ which are locally 
contained in smooth surfaces with distinct tangent spaces at $p$. 
\item[(2)] Multiplicity structures on a smooth curve of multiplicity $\leq 4$ 
near $p$. 
\end{enumerate}

Regarding organization, we review $\rdA_n$ singularities and their analytic 
equations in Section 2, proving the recognition theorem, Theorem \ref{ratrecog}. 
In Sections 3 - 4 we solve Problem \ref{uno} in the cases (1) and (2) listed above. 
Finally in Section 5 we prove Theorem \ref{reduce}, which shows how to compute $\Pic S$ 
and give examples. 

\section{Analytic equations of rational double points}

In this section we briefly review rational double points of type 
$\rdA_n$ and some results about analytic change of coordinates. 

\subsection{$\rdA_n$ singularities}

An $\rdA_n$ surface singularity has local analytic equation 
$xy - z^{n+1}$, thus is is analytically isomorphic to $\Spec(R)$ with 
$R = k [[x,y,z]]/(xy-z^{n+1})$. The resolution of this singularity is well
known \cite[5.2]{Zeuthen}: 
an $\rdA_1$ resolves in a single blow-up with one rational exceptional curve
having self-intersection 
$-2$; an $\rdA_2$ resolves in one blow-up but with two $(-2)$-curves meeting
at a point. 
For $n \ge 3$, blowing up with new variables $x_1 = x/z, y_1 = y/z$ gives
two exceptional curves, namely $E_{x_1}$ 
defined by $(x_1,z)$ and $E_{y_1}$ defined by $(y_1,z)$, meeting
transversely at an $\rdA_{n-2}$ at the origin. 
Blowing up and continuing inductively, the singularity unfolds and we
obtain a resolution with exceptional divisors forming a chain 
of $n$ rational $(-2)$-curves meeting pairwise transversely. 
We will adopt the convention, identifying a curve with its strict transform, 
that $E_1=E_{x_1}, E_2=E_{x_2}, \dots, E_n = E_{y_1}$.

To calculate $\Cl R$, identify a curve $C$ with the sequence 
$(\tilde{C}.E_1, \tilde{C}.E_2 \dots)$ of intersection numbers of its
strict transform with the exceptional curves. 
Then $\Cl R$ is the quotient of the free abelian group on the exceptional
curves with relations given by the fact 
that the exceptional curves themselves correspond to the trivial class
~\cite[$\S 14$ and $\S 17$]{lipman}. 
Let $(u_j)$ be the ordered basis for the free group; then the relations
for an $\rdA_n$ singularity are 
\[
-2u_1+u_2, u_1-2u_2+u_3, \dots, u_{n-1} - 2u_n
\]
so that $\Cl R \cong \Z/(n+1)\Z$ generated by $u_1$ and 
satisfying $u_j=ju_1$ for all $j$. 
\begin{ex}\label{image}{\em 
Let $R=k[[x,y,z]]/(xy-z^{n+1})$ be the complete local ring of an 
$\rdA_n$ surface singularity. Under the identification 
$\Cl (R) \cong \Z / (n + 1) \Z$, we identify the following classes:
\begin{enumerate}
\item[(a)] The class of the curve $D_1$ given by $(x,z)$ is $1$.
\item[(b)] The class of the curve $D_2$ given by $(y,z)$ is $-1$. 
\item[(c)] For $1 \leq r \leq n$, the class of the curve $(x-z^{n-r+1},y-z^r)$ is $r$.
\end{enumerate}
Parts (a) and (b) are contained in \cite[Prop. 5.2]{Zeuthen} and part (c) is \cite[Rem. 5.2.1]{Zeuthen}, 
where the class considered is $(x-az^{n-r+1},y-a^{-1}z^r)$ for a unit $a$. 
To save a change of coordinates at the end of a calculation, we will often apply part (c) to the curve 
$(x-uz^{n-r+1},y-vz^r)$ in the ring $k [[x,y,z]]/(xy-uvz^{n+1})$ with $u,v$ units.
\em}\end{ex}

\subsection{Analytic coordinate changes}

We will use coordinate changes in $k[[x,y,z]]$ to recognize the structure of surface singularities.  
Let $R = k[[x_1,\dots,x_n]]$ be the ring of formal powers series over a field $k$ 
with maximal ideal $\fm$. 
A {\it change of variables} for $R$ is an assignment $x_i \mapsto x_i^\prime \in \fm$ 
inducing an automorphism of $R$. An assignment $x_i \mapsto x_i^\prime$ induces 
an automorphism if and only if induced maps $\fm / \fm^2 \to \fm / \fm^2$ is an 
isomorphism if and only if the matrix $A$ of coefficients of linear terms in 
the $x_i^\prime$ is nonsingular (this is noted by Jaffe \cite[Prop. 3.2]{J2} when $n=2$). 
Examples include multiplication of variables by units and translations of 
variables by elements in $\fm^2$. 

The following lemma allows us to take roots in power series rings. 

\begin{lem}\label{comproots} Let $(R,\fm)$ be a complete local domain, and
let $n$ be a positive integer that is a unit in $R$. 
If $a_0 \in R$ is a unit and $u \equiv a_0^n$ mod $\fm^k$ 
for some fixed $k > 0$, then there exists $a \in R$ such that $a^n = u$
and $a \equiv a_0$ mod $\fm^k$.
\end{lem}

\begin{proof} \cite[Prop. 3.4]{BN2}. \end{proof}

\begin{lem}\label{xyfactor}
Let $R = k[[x,y]]$ with maximal ideal $\fm \subset R$. 
For $f \in \fm^3$, there is a change of coordinates $X,Y$ such that 
\[
xy+f = XY
\]
and $X,Y$ may be chosen so that $x \equiv X$ mod $\fm^2$ and $y
\equiv Y$ mod $\fm^2$. 
\end{lem}

\begin{proof} \cite[Prop. 3.6]{BN2} or \cite[I, Ex. 5.6.1]{AG}.
\end{proof}

We frequently encounter the local equation in the following Proposition.

\begin{prop}\label{2birds}
Given units $u,v \in k[[x,y,z]]$ and integers $m\ge 1, n\ge 3$, let $R = k[[x,y,z]]/(xy-uyz^m-vx^n)$. 
Then $\Spec R$ is an $\rdA_{mn-1}$ singularity and the curve class 
$(x,z)$ (resp. $(x,y)$) maps to $1$ (resp. $mn-m$) via the isomorphism $\Cl R \cong \Z / mn \Z$. 
\end{prop}

\begin{proof} 
In setting $X = x - u z^m$, the equation $xy-uyz^m-vx^n$ becomes
\begin{equation}\label{2a}
Xy + v(X+uz^m)^n=
Xy-vX\ubm{(X^{n-1}+nX^{n-2}uz^m + \dots +
nu^{n-1}z^{m\cdot(n-1)})}{\alpha} - vu^n z^{mn}.
\end{equation}

For $\alpha$ as shown, set $Y=y-v\alpha$ to obtain $XY - v u^n z^{mn}$. 
Absorbing $vu^n$ into either $X$ or $Y$ brings the equation to the standard 
form for an $\rdA_{mn-1}$ singularity, hence we have $\Cl R \cong \Z / mn \Z$ 
from the previous section. Moreover $(x,z) = (X+uz^m, z) = (X,z)$ gives the 
canonical generator $1$ in $\Cl R = \Z / mn \Z$ by Example \ref{image} (a). 

For the second curve, write $(x, y) = (X + uz^m, y) = (X + uz^m,Y + v \alpha)$. Then 
\[
(X+uz^m)^n=X \alpha + (uz^m)^n = (X+uz^m) \alpha - uz^m \alpha + (uz^m)^n
\]
by definition of $\alpha$ so that $(uz^m) \alpha \equiv (uz^m)^n$ mod $(X+uz^m)$. 
Since the quotient ring modulo $X+uz^m$ is an integral domain in which $uz^m$ is nonzero, 
we see that $\alpha \equiv (uz^m)^{n-1}$ mod $(X + u z^m)$ and therefore  
$(x,y) = (X + u z^m,Y + v u^{n-1} z^{mn-m})$ which corresponds to 
$mn-m \in \Z / mn \Z$ by Example \ref{image} (c).

\end{proof}

\subsection{Recognizing $\rdA_n$ Singularities}\label{singtype}

We develop an algorithm to identify the $\rdA_n$-singularity type defined by a power series $F$ in three variables defining a double point with nondegenerate tangent cone, so that the degree-$2$ part is not a square. In this case it is known \cite[Thm. 4.5]{Zeuthen} that $F$ defines $\rdA_n$ singularity for some $n$ or 
$F$ factors, which we interpret as $n = \infty$. Our goal is to identify the answer by 
inspection if possible. Such $F$ can be written
\[
F = \sum_{i+j+k > 1} c_{i,j,k} x^i y^j z^k \in \fm^2 \subset k[[x,y,z]]
\]
with $c_{1,1,0} = 1, c_{2,0,0}=c_{0,2,0} = 0$.

Let $A$ be the sum of all terms satisfying $i,j > 0$ and $i+j+k > 2$. 
Then $A = xy B$ with $B \in \fm$ and the remaining terms of $F$ fall into three 
categories: (a) $i=j=0$, which we write as $h(z) \in k[[z]]$, 
(b) $i=0, j>0$, which we can write as $\sum_{j=1} y^j g_j (z)$ with $g_j \in k[[z]]$ 
and (c) $j=0, i>0$ which can be written as $\sum_{i=1} x^i f_i (z)$ with $f_i \in k[[z]]$. 
With these choices $F$ becomes
\begin{equation} 
F = xy + h(z) + \sum_{i=1} x^i f_i + \sum_{j=1} y^j g_j + xy B
\end{equation}
where $h, f_i, g_j \in k[[z]]$, $B \in \fm$ and $\ord g_2 > 0$. 
Letting $u = 1 + B$ we can drop the 
last term at the expense of multiplying the $xy$ term by the unit $u$: now let 
$X = ux$ and replace $f_i$ with $(u^{-1})^i f_i$ to obtain 
\begin{equation}
F = xy + h(z) + \sum_{i=1} x^i f_i + \sum_{j=1}  y^j g_j
\end{equation}
with $g_j \in k[[z]]$ and $f_i (z) = z^{r^i} u_i$ with $u_i$ a unit. 
To determine the singularity type, we may assume $r_1 < \infty$ or $s_1 < \infty$, since $f_1 = g_1 = 0$ gives an $\rdA_{h-1}$ with $h = \ord h(z)$. 
We make one more simplification. 
Set $X = x + g_1$ to obtain 
\[
F = Xy + h(z) + \sum_{i=1} (X - g_1)^i f_i + \sum_{j=2} y^j g_j.
\]
Regrouping the $f_i$ by powers of $X$ after expanding the powers of $(X - g_1)$ 
we arrive at 
\begin{equation}\label{form}
F = xy + h(z) + \sum_{i=1} x^i f_i + \sum_{j=2} y^j g_j.
\end{equation}
where the $f_i, g_j$ are equal to a power of $z$ times a unit, $0 < \ord f_1 < \infty$ and $0 < \ord g_2$.  
Since we only make variable changes which fix $z$, the crux of the matter is to understand the case when $h=0$. 

\begin{lem}\label{induct} Consider the equation 
\begin{equation}\label{one}
F= xy + \sum_{i=1}^\infty x^i f_i + \sum_{j=2}^\infty y^j g_j 
\end{equation}
where $f_i, g_j$ are powers of $z$ up to units, 
(a) $0 < \ord f_1 < \infty$ and (b) $\ord f_2$ or $\ord g_2 > 0$. 
Set $m=\min\{\ord f_1^j g_j\}$ 
and write $f_1^2 g_2 - f_1^3 g_3 + \dots=z^m \cdot \delta$. 
Then the variable change $X=x, Y=y+f_1$ yields
\begin{equation}\label{two}
F = z^m \cdot \delta +XY+\sum_{i=2}^\infty X^i F_i + \sum_{j=1}^\infty Y^j G_j 
\end{equation}
where $F_i,G_j$ are powers of $z$ up to units such that
\begin{enumerate}[(i)]
\item[(a)] $0 < \ord G_1 < \infty$;
\item[(b)] $\ord F_2 > 0$ or $\ord G_2 > 0$;
\item[(c)] $M=\min\{\ord G_1^i F_i\}>m$; and
\item[(d)] $\ord G_1 \ge m - \ord f_1$.
\end{enumerate}
\end{lem}

\begin{proof}
Setting $Y =y+f_1$ we have
$$ F =  xY + \sum_2^\infty x^if_i + \sum_2^\infty (Y-f_1)^j g_j.$$
The part of the last sum with degree $0$ in $Y$ is 
$\sum_2^\infty (-1)^jf_1^j g_j = z^m \cdot \delta$ by definition of $\delta$.

We take $F_i = f_i$ for $i \geq 2$ and calculate $G_j$ by gathering terms with like powers of $Y$:
\begin{equation}\label{g1}
G_1 = -2f_1g_2 + 3f_1^2g_3 - \dots = \sum_{k=2}^\infty (-1)^{k-1} k f_1^{k-1} g_k
\end{equation}
and for $j \ge 2$,
\[
G_j = \sum_{k=j}^\infty (-1)^{k-j} {k \choose {j}} f_1^{k-j} g_k.
\]

Thus we see that $F$ takes the form of equation (\ref{two}) and it remains to show that 
$M = \min\{ \ord G_1^i F_i \} > m$. When expanded, each term in $G_1^i F_i = G_1^i f_i$ has the form 
\[
c f_i f_1^{k_1+k_2+\dots + k_i-i} g_{k_1}g_{k_2} \cdots g_{k_i}
\]
where $c$ is a constant and the $k_\ell \ge 2$ are not necessarily distinct. The order of this term is strictly greater 
than $\ord f_1^{k_1} g_{k_1} \geq m$, unless $i=2, k_1=k_2=2$. In the case $i=2,k_1=k_2=2$ we would like 
to see that $\ord f_2 f_1^2 g_2^2 > \ord f_1^2 g_2$, but this follows from the condition that $\ord f_2 > 0$ or $\ord g_2 > 0$. 
Thus $\ord G_1^i f_i > m$ for all $i$ and $M > m$. 

For (d), the order of the $k^\text{th}$ term in sum (\ref{g1}) 
is $(k-1)\ord f_1 + \ord g_k \ge m - \ord f_1$. 
\end{proof}

\begin{thm}\label{ratrecog} For $F$ as in equation (\ref{form}), 
let $m=\min\{ j r_1 + \ord g_j\}$ as above and set $\mu = \min\{ \ord h,  m\}$. 
Let $\delta (m)$ be the coefficient of $z^m$ in $H=h + f_1^2 g_2 - f_1^3 g_3 + \dots$. Then 
\begin{enumerate}
\item[(a)] $F$ defines an $A_n$ singularity with $n \geq \mu-1$ ($n = \infty$ is possible). 
\item[(b)] $F$ defines an $A_{\mu-1}$ singularity if $\delta (\mu) \neq 0$. 
\end{enumerate}
\end{thm} 

\begin{proof} Apply the lemma to $F-h(z)$ and then add $h(z)$ back in to obtain the form  
\[
F=H + \sum_{i=2} X^i F_i + \sum_{j=1} Y^j G_j;
\]
then relabel and repeat. By Lemma \ref{induct} (c), $m$ is strictly increasing. Note that after each change of variables $Y=y+f_1$( or $X=x+g_1$), the new variables $(x,Y,z)$ still form a regular system of parameters at the origin. Now, consider two iterations of the algorithm; start with $x,y$ and $f_i, g_j$ and $m$-value $m$; then change to $x,Y$ with $f_i, G_j$ and $m$-value $M>m$, and next to  $X,Y$ with $F_i, G_j$. Then 
\[
\ord F_1 \ge M - \ord G_1 > m - (m-\ord f_1) = \ord f_1
\]
by Lemma \ref{induct}(d), so $\ord f_1$ increases with every change of $x$-variable; 
and similarly for $\ord g_1$. Thus the sequence of variable changes forms a Cauchy sequence and moreover in the limit the terms $f_1$ and $g_1$ both vanish. Therefore the expression becomes
$$XY + H(z) +\sum_{i=2}^\infty X^i F_i + \sum_{j=2}^\infty Y^j G_j,$$
and applying \cite[Prop. 4.4]{Zeuthen} after subtracting $H(z)$ brings us to the form 
$$F = XY + H(z).$$
If some $\delta(\mu) \neq 0$, $H$ retains a term of order $\mu$ in every subsequent change of variables because each only involves terms of order $\ge m > \mu$, so $\mu$ stabilizes. 
Therefore in this case the form of $F$ is $XY + \text{ unit }\cdot z^\mu$, 
an $\rdA_{\mu -1}$ singularity. Otherwise $\delta(\mu) = 0$ for every $\mu$ and 
$H \ra 0$ as $\mu \ra\infty$, so $F$ factors.
\end{proof}

\begin{rmk}{\em 
The inductive procedure given in Theorem \ref{ratrecog} and Lemma \ref{induct} 
recognizes an $\rdA_n$ singularity in at most $n$ steps. However condition (b) in Theorem \ref{ratrecog} is an open condition among equations of fixed degree, so the algorithm terminates after only {\it one} step with probability $1$. 
\em}\end{rmk}

\begin{ex}{\em We illustrate the theorem with a few examples. 

(a) Applying Theorem \ref{ratrecog} to $F = xy + x z^2 + y^2 z - z^6$, we have $m = 5, \mu = \min\{5,6\}=5$ 
and $\delta(5)=1 \neq 0$, so $F$ represents an $\rdA_4$ singularity. The variable change $Y=y+z^2$ 
gives $F = xY + (Y-z^2)^2 z - z^6 = xY + Y^2 z - 2 Y z^3 + z^5 - z^6$ so that $H(z) = z^5 - z^6$ has order $5$. 
After the sequence of variable changes suggested, the $z^5$ term survives while the terms involving $x,Y$ 
eventually factor. 

(b) For the singularity given by 
\[
F=xy + x z^4 + y^2 z^6 + y^3 z^2 + y^4 z^{25} + x^2 z
\]
we have $m=\mu=14$ and $\delta (14) = 0$, so we make the variable change 
$Y=y+z^4$ suggested by the theorem. Then we have 
\[
F=xY + (Y-z^4)^2 z^6 + (Y-z^4)^3 z^2 + (Y-z^4)^4 z^{25} + x^2 z
\]
When multiplying this out, the $z^{14}$ term drops out (because $\delta (14) = 0$), 
but that the new incarnation of $F$ has linear $Y$-terms, namely 
\[
f = xY + Y (-2 z^{10} + 3 z^{10}) + \dots + x^2 z = xY + Y z^{10} + \dots + x^2 z
\]
where the dots represent higher power of $Y$ terms. Continuing with $X = x+z^{10}$ gives
\[
f = XY + \dots + X^2 z - 2 X z^{11} + z^{21}
\]
and it becomes clear that we are dealing with an $\rdA_{20}$.

(c) Prop. \ref{2birds} follows readily from Theorem \ref{ratrecog} as (with $x,y$ reversed) we have 
$\mu = mn$ and $\delta(mn)=u^n v \neq 0$, yielding an $\rdA_{mn-1}$ singularity. 
\em}\end{ex}

\section{Two multiple curves intersect at a point}\label{2multcurves}

In this section we give a solution to Problem \ref{uno} when 
$Z = Z_1 \cup Z_2$ is a union of two multiple curves of embedding 
dimension two with respective smooth supports $C_1, C_2$ meeting 
transversely at $p$ under the condition that $Z_1$ and $Z_2$ 
do not share the same Zariski tangent space at $p$. 
In other words, we consider the following two cases: 

\begin{enumerate}
\item No Tangency: $C_1$ is not tangent to $Z_2$ and $C_2$ is not 
tangent to $Z_1$.
\item Mixed Tangency: $C_1$ is tangent to $Z_2$ but $C_2$ is not 
tangent to $Z_1$.
\end{enumerate}

For each of these we find canonical forms for the local ideals 
(Propositions \ref{case1} and \ref{case2}) and determine the local Picard groups 
at the corresponding fixed singularity on the very general surface containing 
the curve (Propositions \ref{picloc1}, \ref{picloc2a}, \ref{picloc2b} and \ref{picloc2c}). 
The following local algebra lemma will facilitate computing the intersection of ideals 
$I_Z = I_{Z_1} \cap I_{Z_2}$.

\begin{lem}\label{intersect}
Let $R$ be a regular (local) ring. For $a,b,c,d \in R$, assume that $a,c,d$ form a regular 
sequence and that $d \in (a,b)$. Then $(a,b) \cap (c,d) = (ac,bc,d)$. 
\end{lem}

\begin{proof}
Write $d = as + br$ with $s,r \in R$. Since $d$ is a non-zero divisor mod $(a)$, the same 
is true of $r$, so that $a,r$ and $a,b$ also form regular sequences in $R$. 
Since $(a,d) = (a,br)$, the ideals $(a,b)$ and $(a,r)$ are linked by the 
complete intersection $(a,d)$. It follows that 
$(a,b) \cap (c,d)$ is linked to $(a,r)$ by the complete intersection 
$(ac,d) = (a,d) \cap (c,d)$. The inclusion of ideals 
$(ac,d) \subset (a,r)$ lifts to a map of the corresponding Koszul complexes
\[
\begin{array}{ccccccc}
0 & \ra & R & \stackrel{(-d,ac)}{\longrightarrow} & R^2 & \stackrel{(ac,d)}{\longrightarrow} & (ac,d) \\
& & \downarrow \alpha & & \downarrow \beta & & \downarrow \\ 
0 & \ra & R & \stackrel{(-r,a)}{\longrightarrow} & R^2 & \stackrel{(a,r)}{\longrightarrow} & (a,r)
\end{array}
\]
where $\beta(A,B) = (Ac+Bs,Bb)$ and $\alpha(C) = Cbc$. 
By the mapping cone construction for liaison \cite[Prop. 2.6]{ps}, 
the ideal $(a,b) \cap (c,d)$ is the image of $R^3 \to R$ given by 
the direct sum of $\alpha^\vee$ and $(-d,ac)^\vee$, 
so the ideal is $(bc,-d,ac)=(ac,bc,d)$. 
\end{proof}

\begin{ex}{\em
Lemma \ref{intersect} fails if $a,c,d$ do not form a regular sequence 
in $R$, for example $R=k[x,y,z], a=c=x, b=d=y$ 
when $(x,y) \cap (x,y) \neq (x^{2}, yx, y)=(x^{2},y)$. 
\em}\end{ex}

\begin{prop}\label{case1}
Let $Z = Z_2 \bigcup Z_2$ be the union of two multiplicity structures
on smooth curves $C_1,C_2$ meeting at $p$ with respective multiplicities
$m \leq n$. 
Assume $Z_{i}$ is contained in a local smooth 
surface $S_i, i=1,2$, $C_1$ is not tangent to $S_2$ and $C_2$ is not tangent to $S_1$. 
Then there are local coordinates $x,y,z$ at $p$ for which 
$I_{Z_1} = (x,z^m)$, $I_{Z_2} = (y,z^n)$ and 
\[ 
I_{Z}=(xy,yz^m,z^n).
\]
\end{prop}

\begin{proof}
Locally we may assume that $S_{1}$ is given by equation $x=0$ and 
$S_{2}$ is given by equation $y=0$. 
Letting $z=0$ be the equation of a smooth surface containing both 
$C_1$ and $C_2$ near $p$, the lack of tangency conditions imply 
that $x,y,z$ is a regular system of parameters at $p$ and 
we obtain $I_{C_1}=(x,z)$ and $I_{C_2}=(y,z)$. Given that 
$Z_i \subset S_i$ with the multiplicities given, it's clear that 
$I_{Z_1} = (x,z^m)$ and $I_{Z_2} = (y,z^n)$. Taking $a=x, b=z^m, c=y, d=z^n$, 
we have $d \in (a,b)$ because $m \leq n$, so the intersection ideal 
is $(xy,y z^m, z^n)$ by Lemma \ref{intersect}. 
\end{proof}

\begin{prop}\label{picloc1}
For $Z$ as in Proposition \ref{case1} above, the general surface $S$ 
containing $Z$ has an $\rdA_{n-1}$ singularity at $p$ and $C_1$ (resp. $C_2$) maps to $1$ (resp. $-1$) under the isomorphism $\Cl \widehat \O_{S,p} \cong \Z / n \Z$. 
\end{prop}

\begin{proof}
The general element of $I_Z$ has the form $xy+byz^m+cz^n$ with units
$b,c \in \O_{\Pthree,p}$. In the language of Theorem~\ref{ratrecog}, 
$\mu = n$ and $\delta(n) = c\neq 0$, so the singularity is of type $\rdA_{n-1}$.
The first change of variables $X=x+bz^m$ is the only one necessary, giving 
us immediately (up to units) the form $XY - z^n$; furthermore, the ideal defining 
$C_1$ is is $(x, z) = (X-cz^n, z) = (X, z)$ and $I_{C_2} = (y,z)=(Y,z)$, 
so these curves are the canonical generators $\pm 1$ for the group 
$\Cl \widehat \O_{S,p} \cong \Z/ n\Z$ by Examples \ref{image} (a) and (b).


\end{proof}

The mixed tangency case is more complicated. 

\begin{prop}\label{case2}
Let $C_1, C_2$ be smooth curves meeting transversely at $p$, 
$C_i \subset S_i$ local smooth surfaces, and $Z_1 = m C_1 \subset S_1, 
Z_2 = n C_2 \subset S_2$ multiplicity structures. 
Assume that $C_1$ meets $S_2$ transversely and $C_2$ is tangent to
$S_1$ of order $q > 1$. 
Then there are local coordinates $x,y,z$ at $p$ for which 
\[
I_{Z_1}=(x-z^q,z^m), \; \; 
I_{Z_2}=(y,x^n) 
\]
and the intersection ideal $I_{Z_1 \cup Z_2} = I_{Z_1} \cap I_{Z_2}$ takes the form: 
\begin{enumerate}
\item[(a)]\label{singequ2a}
If $m \leq q$, then $I_Z=(xy,yz^{m},x^{n})$. 
\item[(b)]\label{singequ2b}
If $q < m < qn$, then $I_Z=(y(x-z^{q}),yz^{m},x^{n})$. 
\item[(c)]\label{singequ2c}
If $m \geq qn$, then 
$I_Z = (y(x-z^{q}),yz^{m},x^{n} z^{m-qn})$. 
\end{enumerate}
\end{prop}

\begin{proof} 
Let $x=0$ (resp. $y=0$) be a local equation for $S_{1}$ (resp. $S_{2}$). 
Since $C_1$ meets $S_2$ transversely, we can extend $x,y$ to a regular 
sequence $x,y,z$ with $I_{C_1}=(x,z)$. 
Locally $C_2$ meets $S_1$ tangently to order $q > 1$, so we may write 
$I_{C_2}=(x+\alpha,y)$ with $\alpha \in (x,y,z)^q$. 
Now $I_{C_2 \cap S_1} = (x,y,\alpha)$ defines a scheme of length $q$, 
so $\alpha = u z^{q}$ modulo $(x,y)$ for some unit $u$: writing 
$\alpha = u z^q + xf + yg$ we have 
\[
I_{C_2}=(x+\alpha,y)=(x+uz^q+xf+yg,y)=(x(1+f)+u z^{q}, y)
\]
where $(1+f)$ is a unit. Replacing $x$ with $\displaystyle \frac{x(1+f)}{u}+z^q$ we have
\[
I_{Z_1}=(x-z^q,z^m) \;\;\;\;\;\;\; I_{Z_2}=(y,x^n)
\]
and it remains to find the intersection $I_{Z}=I_{Z_{1}} \cap I_{Z_{2}}$.

If $m \leq q$ (including the case $q = \infty \Rightarrow \alpha = 0
\Rightarrow C_{2} \subset S_{1}$), 
then $I_{Z_{1}}=(x,z^{m})$. Apply Lemma \ref{intersect} with 
$a=x,b=z^{m},c=y$ and $d=x^{n}$, we obtain and $I_{Z}=(xy,yz^{m},x^{n})$. 

If $q < m < qn$, then $z^{qn} = z^m \cdot z^{qn-m} \in I_{Z_1}$ and also 
$(x-z^{q})|(x^{n}-z^{qn}) \Rightarrow x^{n}-z^{qn} \in I_{Z_1}$ so $x^n \in I_{Z_1}$. 
Application of Lemma \ref{intersect} with $a=x-z^{q}, b=z^{m}, c=y, d=x^{n}$ 
gives $I_{Z}=(y(x-z^{q}),yz^{m},x^{n})$.

If $qn \leq m$, then we have the telescoping sum 
\[
z^m + z^{m-q} (x - z^q) + x z^{m-2q} (x-z^q) + \dots + x^{n-1} z^{m-nq} (x-z^q) = 
x^n z^{m-qn} \in I_{Z_1}
\] 
and so we can again apply Lemma \ref{intersect} with $a=x-z^{q}, 
b=z^{m}, c=y, d=x^{n} z^{m-qn}$ to obtain $I_{Z}=(y(x-z^{q}),yz^{m},x^{n} z^{m-qn})$. 
\end{proof}

\begin{prop}\label{picloc2a} For $Z = Z_1 \cup Z_2$ as in 
Proposition \ref{case2} (a) with $m \leq q$, the general 
surface $S$ containing $Z$ has a singularity of type $\rdA_{mn-1}$ 
at $p$ and $C_1$ (resp. $C_2$) 
maps to $1$ (resp. $mn-m$) under the isomorphism 
$\Cl \widehat \O_{S,p} \cong \Z / mn \Z$.
\end{prop}

\begin{proof} In view of Prop. \ref{singequ2a} (a), the general surface
$S$ containing 
$Z_1 \cup Z_2$ has local equation $xy - uyz^m - vx^n$ with $u,v$ units in
$\O_{\mathbb P^3,p}$ and $x,y,z$ a regular sequence of parameters. Noting that 
$C_1$ is given by the ideal $(x,z)$ and $C_2$ is given by $(x,y)$, the 
result follows from Proposition \ref{2birds}.
\end{proof}

\begin{prop}\label{picloc2b} 
For $Z = Z_1 \cup Z_2$ as in Proposition \ref{case2} (b) with $q < m < qn$, 
the general surface $S$ containing $Z$ has a singularity of 
type $\rdA_{qn}$ at $p$ and $C_1$ (resp. $C_2$) maps to $1$ (resp. $qn-q$) 
under the isomorphism $\Cl \widehat \O_{S,p} \cong \Z / qn \Z$.
\end{prop}

\begin{proof}
The general surface containing $Z$ has local equation $xy - y z^q - u y z^m - v x^n$ 
for units $u,v \in \O_{\Pthree,p}$ and since $w = 1 + u z^{m-q}$ is a unit we may 
write the equation as $xy - w y z^q - v x^n$ and we can apply Proposition \ref{2birds} 
to see the $\rdA_{qn-1}$ singularity and that $C_1$ with ideal $(x,z)$ corresponds to 
$1 \in \Z / qn \Z \cong \Cl \widehat \O_{S,p}$ while $C_2$ with ideal $(x,y)$ corresponds 
to $qn-q$. 
\end{proof}

\begin{prop}\label{picloc2c}
For $Z = Z_1 \cup Z_2$ as in Proposition \ref{case2} (c) with $qn \leq m$, 
the very general surface $S$ containing $Z$ has a singularity of type $A_{m-1}$ 
and $C_1$ (resp. $C_2$) maps to $1$ (resp. $m-q$) under the isomorphism 
$\Cl \widehat \O_{S,p} \cong \Z / m \Z$.
\end{prop}

\begin{proof}

The general surface containing $Z$ has local equation 
$$xy-yz^{q} + uyz^{m} + vx^{n+1}-vx^nz^{q} + wx^{n} z^{m-qn}$$
for units $u,v,w$. In the notation of Theorem~\ref{ratrecog} with the roles of $x$ and $y$ reversed 
we have $\ord(g_1) = q$ since $q<m$ and $h=0$. However, the value of $\mu$ so far depends on the relative size of $q$ an $m-nq$; therefore, we proceed with the first variable change suggested 
by the algorithm in the theorem, $X=x-z^q+uz^m$. For ease of expression we write 
$z^q-uz^m = \beta z^q$ with $\beta$ a unit, when the equation becomes

\begin{equation}\label{case2c}
Xy + \ubm{v(X+\beta z^q)^{n+1}-v(X+\beta z^q)^nz^{q} + w(X+\beta z^q)^{n} z^{m-qn}}{\alpha}.
\end{equation}

Write $\alpha = (X+\beta z^q)^n [vX - v u z^m + w z^{m-qn}] = X K(X,z)+w\beta^n z^m$ 
for appropriate $K(X,z)$ and set 
$Y=y+K(X,z)$ to obtain $XY +w \beta^n z^m$, showing the $\rdA_{m-1}$ singularity. 
Now, $I_{C_1} = (x,z) = (X,z)$ corresponds to 
$1 \in \Z/m\Z \cong \Cl \widehat \O_{S,p}$ by Example \ref{image}. Finally, 
\[
I_{C_2} = (x,y) = (X+\beta z^q,y) = (X+\beta z^q, Y-K(X,z)) = (X+\beta z^q, Y-K(-\beta z^q, z)).
\]
Evaluating $\alpha$ at $X = -\beta z^q$ gives zero; therefore 
$K((-\beta z^q, z)) = wn\beta^{n-1} z^{m-q}$ and 
$I_{C_2} = (X+\beta z^q, Y-wn\beta^{n-1} z^{m-q})$ corresponds to 
$m-q \in \Z/m\Z$ by Example~\ref{image} (c).
\end{proof}

\section{Multiple structures on a smooth curve}\label{multcurve}

Let $Z$ be a locally Cohen-Macaulay multiplicity structure supported 
on a smooth curve $C \subset \mathbb P^3$. The ideal sheaves $\I_Z + \I_C^i$ 
define one-dimensional subschemes of $Z$ and after removing the embedded 
points we arrive at the Cohen-Macaulay filtration 
\begin{equation}\label{CM}
C=Z_{1} \subset Z_{2} \subset \dots \subset Z_{m}=Z
\end{equation}
from work of Banica and Forster \cite{BF} (see also \cite[$\S 2$]{N}). 
If $Z$ has generic embedding dimension two, then the quotient sheaves 
$\I_{Z_{j}}/\I_{Z_{j+1}} = L_{j}$ are line bundles on $C$ and the 
multiplicity of $Z$ is $m$. In this section we solve Problem \ref{uno} for 
such a multiplicity structure $Z$ at a point $p$ of embedding 
dimension three assuming that $Z_{m-1}$ or $Z_{m-2}$ has local 
embedding dimension two at $p$. This class of curves contains 
all multiplicity structures on $C$ with multiplicity $m \leq 4$. 
First we describe the local ideal of $Z$ at $p$. 

\begin{prop}\label{simple}
Let $C \subset \Pthree$ be a smooth curve and Let $Z$ be a locally 
Cohen-Macaulay multiplicity $m$ structure on a $C$ of generic embedding 
dimension two with filtration (\ref{CM}). Let $p \in Z$ be a point of 
embedding dimension three at which $Z_{m-1}$ has embedding dimension two. 
Then 
\begin{enumerate}
\item[(a)] There are local coordinates $x,y,z$ for which $Z$ has local ideal 
$$
I_Y = (x^2,xy,x z^q - y^{m-1}).
$$
\item[(b)] The general surface $S$ containing $Z$ has an 
$\rdA_{(m-1)q-1}$ singularity at $p$ and and the class of 
$C$ maps to $q \in \Cl \widehat \O_{S,p} \cong \Z / q(m-1) \Z$.       
\end{enumerate}
\end{prop}

\begin{proof}
Consider the exact sequence 
$0 \to I_Z \to I_{Z_{m-1}} \stackrel{\pi}{\to} L \to 0$ near $p$. 
Since $Z_{m-1}$ has embedding dimension two at $p$, it locally lies 
on a smooth surface with equation $x=0$ and $I_C = (x,y)$ for suitable $y$, 
whence $I_{Z_{m-1}} = (x,y^{m-1})$. Now $I_Z$ appears locally as the kernel of a 
surjection $\pi:(x,y^{m-1}) \to \O/(x,y)$. If $\pi(y^{m-1}) = u$ is a unit and 
$\pi(x) = \bar h$ with $h \in \O$, then $x - u^{-1} h y^{m-1} \in \ker (\pi) = I_Z$, 
but then $x - u^{-1} h y^{m-1} \not \in \fm_p^2$ implies $Z$ has embedding 
dimension two at $p$, contrary to assumption.
Therefore we may assume $\pi (y^{m-1}) = \bar h$ for $h \in \fm_p$ and $\pi (x) = 1$ 
in which case $I_Z = \ker (\pi) = (x^2, xy, hx-y^{m-1})$. 
If $h \in \O/(x,y)$ vanishes to order $q > 0$ at $p$, we may write 
$h=uz^{q}$ mod $(x,y)$ where $u$ is a unit and $z$ is a local parameter for $C$ at $p$. 
Absorb $u$ into $x$ gives part (a). 

The local ideal of $Z$ is $(x^2,xy,xz^q-y^{m-1},y^m)$, the generator $y^m$ 
being redundant. For this ideal the result can be found in \cite[Prop. 4.5]{BN2}, 
where our purpose was to show that every subgroup of the completion of an $\rdA_r$ 
rational double point arises as the image of the class group, partially answering 
a question of Srinivas. Indeed, combining with Remark \ref{class}, the associated 
class group is $\Z / (m-1) \Z \subset \Z / q(m-1) \Z$. 
\end{proof}

\begin{prop}\label{structure} 
Let $C \subset \Pthree$ be a smooth curve and Let $Z$ be a locally 
Cohen-Macaulay multiplicity $m$ structure on a $C$ of generic embedding 
dimension two with filtration (\ref{CM}). Fix a point $p \in Z$ of 
embedding dimension three. 
If $Z_{m-2}$ has embedding dimension two at $p$ and $Z_{m-1}$ 
does not, then there are local coordinates $x,y,z$ for which $Z$ has local ideal 
\begin{enumerate}
\item[(a)] 
$(x^2, x y^2, x y z^q - y^{m-1}, xy - u z^w (z^q x - y^{m-2}))$, $u$ a unit and $w \geq 0$, or
\item[(b)] 
$(x^2,x y^2, f x y - (z^q x - y^{m-2}))$
where $f=0$ or $f = u z^w$ for some $w > 0$ and unit $u$.
\end{enumerate}
\end{prop}

\begin{proof}
Use Prop. \ref{simple} (a) to write $I=I_{Z_{m-1}}=(x^2,xy,x z^q - y^{m-2})$. 
At the level of sheaves, $\pi$ factors through 
$\F = \I_{Z_{m-1}} \otimes \O_C / \{\mbox{torsion}\}$, a vector bundle on 
$C$ which has rank two because $Z_{m-1}$ is a generic local complete intersection. 
Working in the  free $\O/(x,y)$-module $F = (I \otimes \O/(x,y))/\{\mbox{torsion}\}$ 
near $p$, 
\[
z^q (x^2) = x (x z^q - y^{m-1}) + y^{m-2} (xy)=0
\]
shows that $x^2$ is torsion, hence zero. Therefore 
$F \cong (\O/(x,y))^2$ is freely generated by $xy$ and $x z^q - y^{m-2}$. 

The kernel of the map $I \to F$ is 
\[
(x^2) + (x,y) I = (x^2, xy^2, xy z^q - y^{m-1}).
\]
and we obtain $I_Z = \ker \pi$ by adding the Koszul relation for the surjection of 
free modules $F \to \O/(x,y)$. 
Surjectivity implies that $\pi(xy)$ or $\pi(z^q x - y^{m-2})$ is a unit in $\O/(x,y)$. 
If $\pi (z^q x - y^{m-2})= 1$ and $\pi (xy) = f \in \O/(x,y)$, the Koszul relation 
is $x y - f (z^q x - y^{m-2})$; here $f = u z^w$ for some unit $u$, since $f=0$ leads to the ideal 
$(x^2,xy,y^{m-1})$ which does not have generic embedding dimension two;
this gives the ideal in part (a).
Otherwise take $\pi (xy)=1$ and $\pi (z^q x - y^{m-1}) = f$ 
where $f=0$ or $f=u z^w$ for some $w > 0$ and unit $u$, when the Koszul relation is 
$f xy - (z^q x - y^{m-2})$, giving ideal in part (b).
\end{proof}

\begin{rmk}\label{existence}{\em
Propositions \ref{simple} and \ref{structure} give a local description of ideals 
of certain multiplicity structures $Z$ on a smooth curve $C$. Using the ideal 
in the Proposition to define multiple curves in $\mathbb A^3$, one can obtain 
global examples by taking the closure in $\Pthree$. 
When $C \subset \Pthree$ is a line and $m \leq 4$, all
such global structures have been classified \cite{N,NS}.
\em}\end{rmk}

\begin{prop}\label{step2} 
Let $Z$ be a multiplicity-$m$ structure on a smooth curve $C$ with local ideal at 
$p$ as in Prop. \ref{structure}(a). Then at $p$ the general surface 
$S$ containing $Z$ has an $\rdA_{(m-2)(q+w)+w-1}$ singularity and 
$C$ has class $q+w \in \Z / ((m-2)(q+w)+w) \Z \cong \Cl \widehat \O_{S,p}$.    
\end{prop}

\begin{proof}
The very general surface $S$ of sufficiently high degree containing $Z$ has equation 
\[ 
xy + ax^2 + bxy^2 + cy^m + dxyz^q - dy^{m-1} -uxz^{q+w} + uy^{m-2}z^w 
\] 
for units $a,b,c,d,u$ in the local ring at the origin $p$ by assumption. 

We first apply the preparation steps in the first part of subsection~\ref{singtype} to bring the equation into a form recognizable by Theorem~\ref{ratrecog}, beginning with changing variables to $Y=y+ax$. Expanding and gathering $xY$-terms brings the equation to the form
$$\text{unit} \cdot xY + u_1 x^3 + u_2 Y^m + u_3 x^m + u_4 x^2z^q + u_5 Y^{m-1} + u_6 x^{m-1} -uxz^{q+w} + uY^{m-2} z^w + u_7 x^{m-2} z^w$$

with units $u_1 = a^2 b, u_2 = c, u_3 = c (-a)^m, u_4 = -ad, u_5 = -d, u_6 = -d (-a)^m$ and $u_7 = u (-a)^{m-2}$. Applying Theorem~\ref{ratrecog} we have $r_1 = q+w$, $\mu = (m-2)(q+w) + w$, and 
$\delta( (m-2)(q+w) + w)\neq 0$ for general choice of units, so the singularity type is 
$\rdA_{ (m-2)(q+w) + w-1}$.

To determine the class of $C$ in the completed local ring, we will look at the resolution of the singularity and determine which exceptional curve meets the strict transform of $C$. On the patch 
$Z=1$ on the first blowup, the singularities must lie on the exceptional 
locus $z=0$. This gives the equation (recycling the symbols $x$ and $y$) $xy+ax^2=0$; 
partials similarly give $x=0$ and $y+2ax=0$, so the blown-up surface is singular 
only at the origin on this patch. On $X=1$ the exceptional locus has 
equations $x=0, y=0$, which is smooth, and similarly on the other patch. 
This situation persists until we get to the $(q+w)^{\rm th}$ blow-up, 
which on the patch $Z=1$ has equation 
$$xy + ax^2 - ux + bxy^2z^{q+w} + dxyz^q + dy^{m-1}z^{(q+w)(m-3)} + uy^{m-2}z^{(q+w)(m-4)+w}.$$ 
This surface is smooth at the origin and singular at $(0,u,0)$. 
Changing variables to $y'=y+u$ produces an equation of the form
$$xy' + \text{ (terms of order at least $2$ in $x$ and $y'$ times powers of $z$)} + \text{ unit }\cdot z^{(q+w)(m-4)+w}.$$
As in subsection~\ref{singtype}, this becomes $XY - Z^{(q+w)(m-4)+w}$ where 
the variable changes to obtain $X$ and $Y$ do not affect $z$ and then $Z$ is 
a unit times $z$.

To determine the class of $C$, note that its strict transform passes through the 
origin all the way to the $(q+w)^{\rm th}$ blowup, at which point it still passes 
through the origin but misses the singular point. This gives $C$ the class $q+w$ 
in the complete local Picard group $\Z/((q+w)(m-2)+w)\Z$. As $C$ generates the class group of the original singular point, the order of this group depends on the greatest common divisor of $q+w$ and $(q+w)(m-2)+w$.

\end{proof}

\begin{prop}\label{case2b}
For $p \in Z$ as in Prop. \ref{structure}(b) with $C = \supp Z$, 
let $S$ be the general surface containing $Z$. 
Then locally the equation of $S$ at $p$ has the form
\[
F = a x^2 + b x y^2 + (fxy - (z^q x - y^{m-2}))
\]
for local parameters $x,y,z$ general units $a,b \in \O_{\Pthree,p}$ 
and $f = u z^w$ for some $w > 0$ (interpret $w = \infty$ as $f=0$) and 
$C$ has order $m-2$ in $\Cl \widehat \O_{S,p}$. 
Furthermore 
\begin{enumerate} 
\item If $m=4$, then $S$ has an $\rdA_{2q-1}$ singularity at $p$ and 
$C \mapsto q \in \Z / 2q \Z \cong \Cl \widehat \O_{S,p}$.  

\item If $m > 4, q = 1$, then $S$ has an $\rdA_{m-3}$ singularity at $p$
and $C \mapsto 1 \in \Z/ (m-2) \Z$.

\item\label{fivetwo} If $m=5, q=2$, then $S$ has an $\rdE_6$ singularity at
$p$ and $C \mapsto 1 \in \Z/ 3\Z \cong \Cl \widehat \O_{S,p}$.

\item\label{generalcase} For $m=5$ and $q\ge 3$ or $m\ge 6$ and $q\ge 2$,
the singularity 
of $S$ at $p$ is not a rational double point.

\end{enumerate}
\end{prop}

\begin{proof} The local equation for $S$ follows immediately from Prop. \ref{structure} (b). 
To see that $C$ has order $m-2$, first observe that $(m-2)C$ is Cartier on $S$ at $p$ 
simply because 
\[
(x,F) = (x, a x^2 + b x y^2 + fxy - (z^q x - y^{m-2})) = (x,y^{m-2}).
\]
This shows that the order of $C$ divides $m-2$ and it remains to show the order cannot 
be less. For this, recall that by construction $dC$ has local ideal $(x,y^d)$ for all $d \leq m-2$, 
so we must show that $(x, y^d)$ is not Cartier on $S$ at $p$ for $d < m-2$. By Nakayama's lemma, 
this is equivalent to showing that the $\O/(x,y,z)$-vector space 
\[
\frac{(x,y^d)}{(x,y,z)(x,y^d)+(F)}
\]
has dimension $> 1$, but this is clear because $F \in (x,y,z)(x,y^d)$ for $d < m-2$, so 
the dimension is $2$. 

First assume $m=4$. Taking $e=-1$, the local equation for $S$ at $p$ is 
\[
ax^2 - y^2 - dy^3 + c y^4 + bxy^2 + dxyz^q - fxy + x z^q. 
\]
Setting $x_1 = \sqrt{a}x + y$ and $y_1=\sqrt{a}x - y$ the equation takes the form 
\[
x_1 y_1 + F + \frac{(x_1+y_1)}{2\sqrt{a}} z^q
\]
with $F \in (x_1,y_1)^2 \fm$. By Lemma \ref{xyfactor} there is a coordinate change 
$X,Y$ for which the equation becomes $XY + (AX+BY) z^q$ where $A,B$ are units for 
general choices of $a,b,c,d$. Making the elementary transformation $X_1 = X+Bz^q$ 
and $Y_1 = Y + A z^q$ brings the equation to the form $X_1 Y_1 - AB z^{2q}$ displaying 
the $\rdA_{2q-1}$ singularity. Tracing the class of the supporting curve we have 
\[
(x,y) = (x_1,y_1)=(X,Y)=(X_1-Bz^q, Y_1 - A z^q)
\]
which has class $q \in \Z / (2q) \Z \cong \Cl \widehat \O_{S,p}$ by Example \ref{image} (c). 

Now assume $m > 4$ and $q=1$. Taking $e=-1$, the local equation for $S$ at $p$ is 
\[
xz - y^{m-2} - f x y + a x^2 + b x y^2 + c y^m + d(x y z - y^{m-1})
\]
for units $a,b,c,d \in \O_{\Pthree,p}$. For $Z=z+fy+ax+by^2+dyz^q$ 
and unit $u = 1 + d y - c y^2$ this takes the form $xZ - u y^{m-2}$, 
the equation of an $\rdA_{m-3}$ singularity. The class of the curve with 
ideal $(x,y)$ is $1 \in \Z / (m-2) \Z$ by Example \ref{image} (a).

The case $m=5$ and $q=2$ has a different flavor: 
Write the equation for $S$ as 
$$x^2 + 2axy^2+by^5+2cxyz^2 - cy^4 + \ubm{2uxyz^w}{\text{or
}0}-2vxz^2+2vy^3;$$
let $x_1=x+y^2+cyz^2+uyz^w-vz^2$, and note that the expression takes on
the form 
$$x_1^2 + \alpha y^3 + \beta y^2z^2 + \gamma z^4,$$
where $\alpha, \beta, \gamma$ are units. We may assume $\alpha=1$; rewrite
this expression as 
\[
x_1^2 + \left(y + \frac{\beta}{3} z^2\right)^3 + \gamma^\prime z^4
\]
where $\gamma^\prime$ is another unit. Taking $y_1= y + \frac{\beta}{3}
z^2$ shows that $S$ has an $\rdE_6$ 
singularity at the origin with $\Cl \widehat \O_{S,p} \cong \Z/ 3\Z$. (The ideal for
the curve $C$ has 
become $(x_1-\gamma z^4,y)$.)

For the case $m=5, q\ge 3$ in part (\ref{generalcase}), a calculation
entirely analogous to the previous one gives the form $X^2+Y^3 + aYZ^4 + bZ^6,$
$a,b$ units, and the ideal of $C$ has the form $(X+i\sqrt{b}Z^3, Y)$.
The first blow-up of this surface on the patch (recycling variables as
usual) has equation
$X^2 + Y^3z + aYz^3 + bz^4$, which is not the equation for a rational
double point, since it is 
congruent to a square mod $\fm^4$ (see the classification in \cite[\S 24]{lipman}). 
Therefore the original singularity is not a rational double point, since
the resolution of a rational 
double point only involves other rational double points.

Finally, for $m\ge 6, q\ge 2$ in part (\ref{generalcase}), after the first
algebraic step of completing the square, we see that the equation for $S$ 
is congruent to a square $\mod \fm^4$, so again the singularity is not a 
rational double point.
\end{proof}

\section{Global Picard groups of normal surfaces} 

In this section we give a formula for the Picard group of very general high degree 
surfaces containing a fixed curve $Z$ with at most finitely many points of embedding 
dimension three. The solution to Problem \ref{uno} is required to apply the formula 
and we illustrate this with the examples worked out in the previous two sections. 

\begin{thm}\label{reduce}
Let $Z \subset \mathbb P^3$ be a closed one-dimensional subscheme with curve 
components $Z_1, \dots Z_r$ having respective supports $C_i$ and suppose that the set 
$F$ of points where $Z$ has embedding dimension three is finite. 
If $S$ is a very general surface of degree $d \gg 0$ containing $Z$ with plane section $H$, 
then
\begin{enumerate}
\item $S$ is normal and $\Cl S$ is freely generated by $H$ and the $C_i$. 
\item The Picard group of $S$ is  
\begin{equation}\label{picgroup}
\Pic S = \bigcap_{p \in F} \ker (\Cl S \to \Cl \O_{S,p}) \cap \langle Z_1, Z_2, \dots, Z_r,
H \rangle \subset \Cl S.
\end{equation}
\end{enumerate}
\end{thm}

\begin{proof} Part (1) is \cite[Thm. 1.1]{BN}. It follows from sequence (\ref{jaffe}) in the introduction that 
\begin{equation}\label{kernels}
\Pic S = \bigcap_{p \in {\rm {Sing}} S} \ker(\Cl S \to \Cl \O_{S,p}).
\end{equation}
Along with the fixed singularities $F$, which forcibly 
lie on every surface $S$ containing $Z$, there are {\em moving singularities} $p$, which 
vary with the surface and lie on exactly one component $Z_i$ of multiplicity $m_i > 1$ 
\cite[Prop. 2.2]{BN}: 
these are $\rdA_{m_i-1}$ singularities and the corresponding map 
$\Cl S \to \Cl \widehat \O_{S,p} \cong \Z / m_i \Z$ sends $C_i$ to $1$ and the remaining 
$C_i$ to $0$, therefore the corresponding kernel is 
$\langle C_1, C_2, \dots, m_i C_i = Z_i, C_{i+1}, \dots, C_r, H \rangle$. The intersection of these 
subgroups for $1 \leq i \leq r$ is $\langle Z_1, Z_2, \dots, Z_r, H \rangle$ which gives 
equation (\ref{picgroup}) provided there is at least one component $Z_i$ of multiplicity $m_i > 1$. 
If there are no components of multiplicity $m_i > 1$, then $Z$ is reduced and there are no moving 
singularities: here formula (\ref{picgroup}) still works because $Z_i = C_i$ for each 
$1 \leq i \leq r$ and hence 
$\langle Z_1, \dots, Z_r, H \rangle = \langle C_1, \dots, C_r, H \rangle = \Cl S$.
\end{proof}

We first note some easy special cases. 

\begin{cor}\label{nofixed} Let $Z$ and $S$ be as in Theorem \ref{reduce}. Then 
\begin{enumerate}
\item[(a)] If $Z$ is reduced of embedding dimension at most two, then $\Pic S = \Cl S$. 
\item[(b)] If $Z$ is reduced, then $\Pic S = \bigcap_{p \in F} \ker (\Cl S \to \Cl \O_{S,p})$.
\item[(c)] If $Z$ has embedding dimension $\leq 2$, then 
$\Pic S = \langle Z_1, \dots, Z_r, H \rangle$.
\end{enumerate}
\end{cor}

\begin{proof} (a) Here $F$ is empty and $Z_i = C_i$ for each $i$, so part (2) of the theorem 
says that $\Pic S$ is generated by $H$ and the $C_i$, which is exactly $\Cl S$ by 
part (1). (b) Here $Z_i = C_i$ again, so $\langle Z_i, H \rangle = \Cl S$. 
(c) Here again $F$ is empty. 
\end{proof}

\begin{rmk}{\em 
We see that the Picard groups aren't very interesting 
for nicely behaved curves, which explains why we have focused on non-reduced base locus $Z$. 
For example, if $Z$ has at worst nodes, then by 
Cor. \ref{nofixed} (a) the Picard group of $S$ is freely generated by $H$ and 
the components of $Z$, extending Lopez' theorem \cite[II, Thm. 3.1]{L}. 
\em}\end{rmk}

\begin{ex}\label{lines}{\em 
To see what can happen to smooth curves intersecting at a point, 
consider the simplest case when $Z = \bigcup_{i=1}^r L_i$ is the union of $r$ lines 
passing through $p$.  
\begin{enumerate}
\item [(a)] If $r=2$, then $S$ is smooth at $p$ and $\Pic S$ is freely generated by the 
two lines and $H$ by Cor. \ref{nofixed} (a). 
\item [(b)] If $3 \leq r \leq 5$ and the lines are not coplanar, then $p$ is a fixed singularity, 
but a mild one. Even when the lines are in general position with respect to containing $p$, 
$S$ has an $\rdA_1$ singularity at $p$ and the map $\Cl S \to \Cl \O_{S,p}$ takes each line 
to $1 \in \mathbb Z / 2 \mathbb Z$. Therefore the Picard group is 
\[
\Pic S = \{\sum a_i L_i + b H: 2 | \sum a_i \}
\]
in this case. We had worked out the case $r=4$ in \cite[Ex. 1.4]{BN}. 
\item[(c)] If $r > 5$ and the lines are in general position, then by \cite[Cor. 5.2]{BN2} 
$p$ is a non-rational singularity and the local class group $\O_{S,p}$ contains an 
Abelian variety. Moreover the images of the lines are involved in no relations 
in $\O_{S,p}$, so that $\Pic S = \langle H \rangle$. 
\item[(d)] There are many ways that the lines can lie in special position. 
We have not explored all of them, but we did work out the case of $r$ planar lines 
$L_1, \dots, L_r$ through $p$ union a line $L_0$ not in the plane, this configuration resembles 
a pinwheel \cite[Ex. 5.3 (b)]{BN2}. Here the point $p$ is an $\rdA_{r-1}$-singularity on $S$ and 
the map $\Cl S \to \widehat \Cl \O_{S,p} \cong \mathbb Z / r \mathbb Z$ sends $L_0$ to $1$ and 
the other lines $L_i$ to $-1$. 
Therefore $\Pic S = \langle r L_0, L_0 + L_1, L_0 + L_2, \dots L_0 + L_r, H \rangle$.
\end{enumerate}
\em}\end{ex}

Now we consider examples in which $Z$ is non-reduced, but the set 
$F$ of embedding dimension three points is non-empty.  

\begin{ex}\label{mline}{\em 
Consider the very general high degree surface $S$ containing 
a locally Cohen-Macaulay $m$-structure $Z$ of generic embedding dimension two supported on a line $L$. 
\begin{enumerate}
\item[(a)] If $Z$ has embedding dimension two at each point (always true if $m=2$), 
then $\Pic S = \cyc{H, Z}$ by Cor. \ref{nofixed} (c).
\item[(b)] If the underlying $(m-1)$ structure has embedding dimension two but $Z$ itself does not, 
then $S$ has an $\rdA_{(m-1)q-1}$-singularity at $p$ for some $q > 0$ and the restriction map 
$\Cl S \to \widehat \Cl \O_{S,p} \cong \mathbb Z / q(m-1) \mathbb Z$ sends $L$ to 
$q \in \mathbb Z / q(m-1) \mathbb Z$ by Prop. \ref{simple}. Applying Theorem \ref{reduce} (c) we have 
\[
\Pic S = \langle m L, H \rangle \cap \langle (m-1) L, H \rangle = \langle m(m-1) L, H \rangle. 
\]
For example, the very general surface $S$ containing a typical triple line $Z$ supported on $L$ has 
Picard group $\Pic S = \langle 6 L, H \rangle$. 
\item[(c)] If the underlying $(m-2)$-structure has embedding dimension two and the underlying 
$(m-1)$-structure does not, the story is more complicated because there are two possibilities 
for the local ideal of $Z$ at $p$ by Prop. \ref{structure}. 
In the second of these (see Prop. \ref{case2b} (b)) $L$ has order $m-2$ in $\Cl \O_{S,p}$, 
so Theorem \ref{reduce} gives $\Pic S = \langle m L, H \rangle \cap \langle (m-2) L, \O_S (1) \rangle = \langle LCM(m,m-2) L, H \rangle$. 
The actual singularity may be an $\rdA_n$, an $\rdE_6$ or even irrational. 
\end{enumerate}
\em}\end{ex}

\begin{ex}{\em 
In section \ref{2multcurves} Consider the very general high degree surface $S$ 
containing a union of two multiple lines $Z_1, Z_2$ supported on $L_1, L_2$. 
\begin{enumerate}
\item[(a)] If $Z_1 \cap Z_2 = \emptyset$, then $\Pic S = \langle Z_1, Z_2, H \rangle$ by Cor. \ref{nofixed} (c). 
\item[(b)] Now suppose $I_{Z_1} = (x,z^m)$ and $I_{Z_2}=(y,z^n)$ with $m \leq n$ so that 
$Z_1 \cap Z_2$ is a length $n$ subscheme supported at $p=(0,0,0,1)$. 
By Props. \ref{case1} and \ref{picloc1}, $p$ is an $\rdA_{n-1}$ singularity of $S$ and 
the restriction map $\Cl S \to \Cl \widehat \O_{S,p} \cong \mathbb Z / n \mathbb Z$ 
takes $L_1, L_2$ to $1,-1$. Taking the kernel of this map we find that 
$\Pic S = \langle n L_1, L_1 + L_2, H \rangle$. 
\item[(c)] Now replace $Z_1$ with the multiple line with having $(x^m,z)$. The support of 
$Z_1 \cup Z_2$ is the same as the last example, but now $L_2$ is contained in the plane 
$S_1: \{z=0\}$ containing $Z_1$, so $L_2$ has order of tangency $q = \infty$ to $S_1$. 
According to Prop. \ref{picloc2a}, $S$ has an $\rdA_{mn-1}$ singularity at $p$ and 
the restriction $\Cl S \to \Cl \widehat \O_{S,p} \cong \mathbb Z / mn \mathbb Z$ 
takes $L_1$ to $1$ and $L_2$ to $mn-m$. 
Therefore $\Pic S = \langle mn L_1, L_2 - (mn-m) L_1, H \rangle$.  
\end{enumerate} 
\em}\end{ex} 

\begin{ex}{\em
These results can be used in combination, so we close with an example illustrating several 
behaviors at once. Start with three non-planar lines $L_1, L_2, L_3$ meeting at $p_1$. 
Let $Z_4$ be a $4$-structure on a line $L_4$ intersecting $L_1$ at $p_2 \neq p_1$, and assume 
that $Z_4$ is contained in a smooth quadric surface $Q$ which is tangent to $L_1$. 
Let $Z_5$ be a $3$-structure on a line $L_5$ which intersects $L_2$ in a reduced point $p_3 \neq p_1$, 
and suppose that $Z_5$ has at least one point $p_4 \neq p_3$ of embedding dimension three. 
Finally, let $Z_6$ be a double line supported on $L_6$ which intersects $Z_5$ at a point 
$p_5 \neq p_4, p_3$ and assume that $L_6$ intersects a local surface $S_5$ defining 
$Z_5$ in a double point. Finally let $Z = L_1 \cup L_2 \cup L_3 \cup Z_4 \cup Z_5 \cup Z_6$ 
and consider the very general surface $S$ of high degree containing $Z$. 

By Theorem \ref{reduce} (a), $\Cl S$ is freely generated by $H$ 
and $L_1, L_2, \dots, L_6$ and to find $\Pic S$ we must compute the kernels 
of the maps $\Cl S \to \Cl \O_{S,p_i}$ for $1 \leq i \leq 5$: 
\begin{enumerate}
\item By Ex. \ref{lines} (b) the kernel at $p_1$ is 
$\langle 2 L_1, L_2-L_1, L_3-L_1, L_4, L_5, L_6, H \rangle$ and 
$S$ has an $\rdA_1$ singularity at $p_1$. 
\item By Prop. \ref{picloc2c} with $m=4, n=1, q=2$, 
the natural restriction map is given by $L_4 \mapsto 1, L_1 \mapsto 2 \in \mathbb Z / 4 \mathbb Z$, so the 
kernel at $p_2$ is $\langle L_1 - 2 L_4, L_2, L_3, 4 L_4, L_5, L_6, H \rangle$ and $S$ has 
an $\rdA_3$ singularity at $p_2$. 
\item By Prop. \ref{simple}, the kernel at $p_3$ is 
$\langle L_1, L_2 + L_3, 3 L_3, L_4, L_5, L_6, H \rangle$.
\item By Ex. \ref{mline} (b) the kernel at $p_4$ is $\langle L_1, L_2, L_3, L_4, 6 L_5, L_6, H \rangle$.
\item By Prop. \ref{picloc2c} with $n=q=2, m=3$, the kernel is 
$\langle L_1, L_2, L_3, L_4, L_6 - 2 L_5, 4 L_5, H \rangle$
\end{enumerate}
Using Hermite Normal Form and Mathematica we compute the intersection of the kernels and 
$\langle L_1, L_2, L_3, 4 L_4, 3L_5, 2L_6, H \rangle$ to be 
$\Pic S = \langle 2L_1, 6L_2, L_3+L_2, 4L_4, 12L_5, 2L_6, H \rangle$.
\em}\end{ex}

\end{document}